\documentclass[a4paper,12pt]{article}
\usepackage{amsmath}
\usepackage{latexsym,amssymb,amsbsy,natbib,apalike,rotating,setspace,amsxtra}
\usepackage[T1]{fontenc}
\usepackage[latin1]{inputenc}
\usepackage[english]{babel}
\usepackage{graphicx,pict2e}
\usepackage{algorithm}
\floatname{algorithm}{Algorithm}

\usepackage{amsthm}
\usepackage{amsfonts}
\usepackage{lastpage}
\usepackage{tabularx}
\usepackage{multirow}
\usepackage{color}

\usepackage{tikz}

\usepackage{verbdef}

\newtheorem{theo}{Theorem}
\newtheorem{prop}{Proposition}

\numberwithin{equation}{section}

\title{Properties of Chromy's sampling procedure}
\date{}

\author{Guillaume Chauvet\thanks{\textit{Univ. Rennes, Ensai, CNRS, IRMAR - UMR 6625}}}

\begin{document}

\maketitle

\begin{abstract}
\noindent \cite{chr:79} proposed a unequal probability sampling algorithm, which enables to select a sample in one pass of the sampling frame only. This is the default sequential method used in the \verb"SURVEYSELECT" procedure of the \verb"SAS" software. In this article, we study the properties of Chromy sampling. We prove that the Horvitz-Thompson is asymptotically normally distributed, and give an explicit expression for the second-order inclusion probabilities. This makes it possible to estimate the variance unbiasedly for the randomized version of the method programmed in the \verb"SURVEYSELECT" procedure.
\end{abstract}

\section{Introduction} \label{sec1}

\noindent \cite{chr:79} proposed a fixed-size unequal probability sampling design which is strictly sequential, in the sense that a sample is selected in one pass of the sampling frame only. This algorithm benefits from a stratification effect, in the sense that the selected units are well spread over the population like with systematic sampling. The drawback is that many second-order inclusion probabilities are zero, making unbiased variance estimation not possible. \\

\noindent \cite{chr:79} therefore proposed to partially randomize the order of the units in the population before applying the sampling algorithm. This randomization is sufficient to guarantee that the second-order inclusion probabilities are positive. The randomized Chromy algorithm is the default sequential method currently available in the \verb"SURVEYSELECT" procedure of the \verb"SAS" software. The method has been extensively used for sample surveys, see for example \cite{mil:hen:gol:kur:kon:que:rib:18}, \cite{rad:con:nun:lac:wu:lew:win:sie:18}, \cite{sch:cur:bou:wil:gla:van:wil:rog:18} and \cite{rus:mye:dan:que:las:rib:19} for recent examples. \\

\noindent So far, the properties of this sampling algorithm have not been fully investigated, and this is the purpose of the current paper. We prove that Chromy sampling is equivalent to ordered pivotal sampling \citep{dev:til:98,cha:12}, in the sense that both algorithms lead to the same sampling design. This leads to the Horvitz-Thompson being consistent and asymptotically normally distributed, under weak assumptions. This also leads to an explicit formula for the second-order inclusion probabilities, making unbiased variance estimation possible for randomized Chromy sampling. \\

\noindent The paper is organized as follows. In Section \ref{sec2}, the notation and the assumptions are given. Chromy sampling and ordered pivotal sampling are introduced in Sections \ref{sec3} and \ref{sec4}. The equivalence between both sampling designs is proved in Section \ref{sec5}, and the properties of Chromy sampling are studied. The results of a small simulation study are given in Section \ref{sec6}. We conclude in Section \ref{sec7}. The proofs are gathered in the Appendix.

\section{Notation and assumptions} \label{sec2}

\noindent We consider a finite population $U$ of $N$ sampling units that may be represented by integers $k=1,\ldots,N$. Denote by $\pi=(\pi_1,\ldots,\pi_N)^{\top}$ a vector of probabilities, with $0<\pi_k<1$ for any unit $k \in U$, and with $n=\sum_{k \in U} \pi_k$ the expected sample size. A random sample $S$ is selected in $U$ by means of a sampling design $p(\cdot)$ with parameter $\pi$, in the sense that the expected number of draws for unit $k$ is $\pi_k$. We let $I_k$ denote the number of times that unit $k$ is selected in the sample, and we note $I=(I_1,\ldots,I_N)^{\top}$. \\

\noindent The set of probabilities $\pi$ may be defined proportionally on some positive auxiliary variable known for any unit in the population, which leads to unequal probability sampling with probabilities proportional to size ($\pi$-ps) \citep[][Section 3.6.2]{sar:swe:wre:92}. The sampling algorithm proposed by \cite{chr:79} may handle $\pi_k$'s greater than $1$, in which case a same unit may be selected several times in the sample. In this paper, we focus on the fairly usual situation when all $\pi_k$'s lie between $0$ and $1$, which means that the sampling design is without replacement. In this case, we may interpret $\pi_k$ as the probability for unit $k$ to be included in the sample, and $I_k$ is the sample membership indicator for unit $k$. We are interested in the total $t_y=\sum_{k \in U} y_k$ of some variable of interest taking the value $y_k$ for unit $k \in U$. The Horvitz-Thompson estimator is
    \begin{eqnarray} \label{sec2:eq1}
      \hat{t}_{y\pi} & = & \sum_{k \in U} \frac{y_k}{\pi_k} I_k .
    \end{eqnarray}

\noindent In order to study Chromy sampling, some additional notation is needed. We let $V_k=\sum_{l=1}^k \pi_l$ denote the cumulated inclusion probabilities up to unit $k$, with $V_0=0$. The integer part of $V_k$ is the largest integer smaller than $V_k$, and is denoted as $V_k^I$. The difference between $V_k$ and its integer part $V_k^I$ is the fractional part, and is denoted as $V_k^F=V_k-V_k^I$. For example, if $V_k=3.6$ we have $V_k^I=3$ and $V_k^F=0.6$, and if $V_k=4.0$ we have $V_k^I=4$ and $V_k^F=0$. \\

\noindent A unit $k$ is a {\em cross-border} if $V_{k-1} \le i$ and $V_{k} > i$ for some integer $i=1,\ldots,n-1$. The cross-border units are denoted as $k_i,~i=1,\ldots,n-1,$ and we note $a_i=i-V_{k_i-1}$ and $b_i=V_{k_i}-i$. The cross-border units define a partition of the population into microstrata $U_i,~i=1,\ldots,n,$ which are defined as
    \begin{eqnarray} \label{sec2:eq2}
      U_i & = & \{k \in U;~k_{i-1} \leq k \leq k_i\} \textrm{ with } k_0=0 \textrm{ and } k_n=N+1.
    \end{eqnarray}
The quantities are presented in Figure \ref{nota} for illustration. \\

    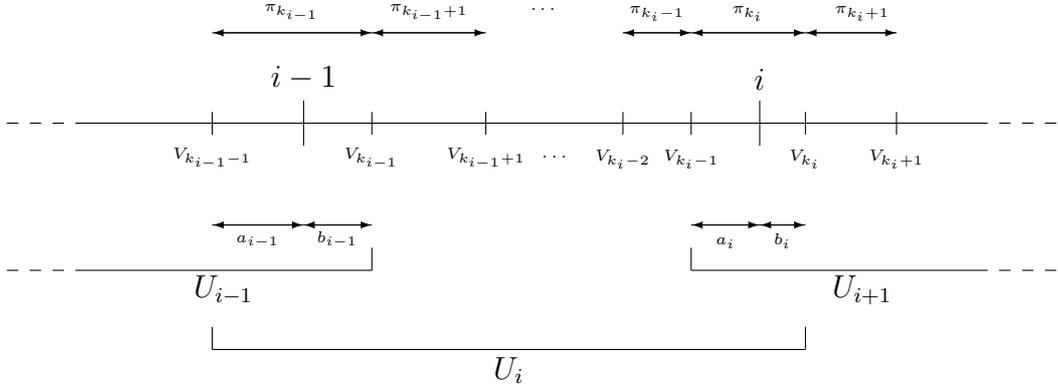
\begin{figure}%[htb!]

    \setlength{\unitlength}{0.03cm}

    \begin{picture}(500,200)

    % Le segment principal
    \put(30,130){\line(1,0){400}}

    \put(20,130){\line(1,0){5}}
    \put(10,130){\line(1,0){5}}
    \put(0,130){\line(1,0){5}}

    \put(435,130){\line(1,0){5}}
    \put(445,130){\line(1,0){5}}
    \put(455,130){\line(1,0){5}}

    % Les bornes enti\`eres
    \put(130,120){\line(0,1){20}} \put(130,145){\makebox(0,0)[b]{$i-1$}}
    \put(330,120){\line(0,1){20}} \put(330,145){\makebox(0,0)[b]{$i$}}

    % Les V(k)
    \put(90,125){\line(0,1){10}}   \put(90,110){\makebox(0,0)[b]{\tiny{$V_{k_{i-1}-1}$}}}
    \put(160,125){\line(0,1){10}}   \put(160,110){\makebox(0,0)[b]{\tiny{$V_{k_{i-1}}$}}}
    \put(210,125){\line(0,1){10}}   \put(210,110){\makebox(0,0)[b]{\tiny{$V_{k_{i-1}+1}$}}}
                                   \put(240,115){\makebox(0,0)[b]{\tiny{$\ldots$}}}
    \put(270,125){\line(0,1){10}}   \put(270,110){\makebox(0,0)[b]{\tiny{$V_{k_{i}-2}$}}}
    \put(300,125){\line(0,1){10}}   \put(300,110){\makebox(0,0)[b]{\tiny{$V_{k_{i}-1}$}}}
    \put(350,125){\line(0,1){10}}   \put(350,110){\makebox(0,0)[b]{\tiny{$V_{k_{i}}$}}}
    \put(390,125){\line(0,1){10}}   \put(390,110){\makebox(0,0)[b]{\tiny{$V_{k_{i}+1}$}}}

    % Les unit\'es frontalières
    \put(90,85){\vector(1,0){40}}  \put(130,85){\vector(-1,0){40}} \put(110,75){\makebox(0,0)[b]{\tiny{$a_{i-1}$}}}
    \put(130,85){\vector(1,0){30}} \put(160,85){\vector(-1,0){30}} \put(145,75){\makebox(0,0)[b]{\tiny{$b_{i-1}$}}}
    \put(300,85){\vector(1,0){30}} \put(330,85){\vector(-1,0){30}} \put(315,75){\makebox(0,0)[b]{\tiny{$a_{i}$}}}
    \put(330,85){\vector(1,0){20}} \put(350,85){\vector(-1,0){20}} \put(340,75){\makebox(0,0)[b]{\tiny{$b_{i}$}}}

    % Les probabilit\'es d'inclusion
    \put(90,170){\vector(1,0){70}}  \put(160,170){\vector(-1,0){70}} \put(125,175){\makebox(0,0)[b]{\tiny{$\pi_{k_{i-1}}$}}}
    \put(160,170){\vector(1,0){50}}  \put(210,170){\vector(-1,0){50}} \put(185,175){\makebox(0,0)[b]{\tiny{$\pi_{k_{i-1}+1}$}}}
                                                                    \put(235,180){\makebox(0,0)[b]{\tiny{$\ldots$}}}
    \put(270,170){\vector(1,0){30}}  \put(300,170){\vector(-1,0){30}} \put(285,175){\makebox(0,0)[b]{\tiny{$\pi_{k_{i}-1}$}}}
    \put(300,170){\vector(1,0){50}}  \put(350,170){\vector(-1,0){50}} \put(325,175){\makebox(0,0)[b]{\tiny{$\pi_{k_{i}}$}}}
    \put(350,170){\vector(1,0){40}}  \put(390,170){\vector(-1,0){40}} \put(375,175){\makebox(0,0)[b]{\tiny{$\pi_{k_{i}+1}$}}}

    % Les micro-strates

    \put(30,65){\line(1,0){130}} \put(160,65){\line(0,1){10}} \put(95,50){\makebox(0,0)[b]{$U_{i-1}$}}
    \put(20,65){\line(1,0){5}}
    \put(10,65){\line(1,0){5}}
    \put(0,65){\line(1,0){5}}

    \put(90,30){\line(1,0){260}} \put(90,30){\line(0,1){10}} \put(350,30){\line(0,1){10}} \put(220,15){\makebox(0,0)[b]{$U_{i}$}}

    \put(300,65){\line(1,0){130}} \put(300,65){\line(0,1){10}} \put(375,50){\makebox(0,0)[b]{$U_{i+1}$}}
    \put(435,65){\line(1,0){5}}
    \put(445,65){\line(1,0){5}}
    \put(455,65){\line(1,0){5}}
    \end{picture}
    \caption{Inclusion probabilities and cross-border units in microstratum $U_i$, for population $U$} \label{nota}
    \end{figure}

\noindent We consider the following assumptions, which are the same than in \cite{cha:leg:18}:
\begin{itemize}
  \item[H1:] There exists some constants $0<f_0$ and $f_1<1$ such that for any $k \in U$:
    \begin{eqnarray} \label{sec2:eq3}
      f_0 \frac{n}{N} & \leq \pi_k \leq & f_1.
    \end{eqnarray}
  \item[H2:] There exists some constant $C_1$ such that:
    \begin{eqnarray} \label{sec2:eq4}
      \sum_{k \in U} \pi_{k} \left(\frac{y_k}{\pi_k}-\frac{t_y}{n} \right)^4 & \leq & C_1 N^4 n^{-3}.
    \end{eqnarray}
  \item[H3:] There exists some constant $C_2>0$ such that:
    \begin{eqnarray} \label{sec2:eq5}
      \sum_{i=1}^{n} \sum_{k \in U_i} \alpha_{ik} \left(\frac{y_k}{\pi_k}-\sum_{l \in U_i} \alpha_{il} \frac{y_l}{\pi_l}\right)^2 & \geq & C_2 N^2 n^{-1},
    \end{eqnarray}
  where for any unit $k \in U_i$ we take:
    \begin{eqnarray*} %\label{sec2:eq6}
    \alpha_{ik} & = & \left\{\begin{array}{ll}
                               b_{i-1} & \textrm{if } k=k_{i-1}, \\
                               \pi_k & \textrm{if } k_{i-1}<k<k_i, \\
                               a_i & \textrm{if } k=k_i,
                             \end{array}
                      \right.
    \end{eqnarray*}
  with the convention that $b_0=a_n=0$.
\end{itemize}

\noindent It is assumed in (H1) that the first-order inclusion probabilities are bounded away from $1$. This is not a severe restriction in practice, since a unit with an inclusion probability close to $1$ is usually placed into a take-all stratum, i.e. the probability is rounded to $1$ and the unit is not involved in the selection process. It is also assumed in (H1) that the first-order inclusion probabilities have a lower bound of order $n/N$, which ensures that no design-weight $d_k=1/\pi_k$ is disproportionately larger than the others. \\
%This is a key assumption to control that the Horvitz-Thompson estimator concentrates asymptotically around its mean value. \\

\noindent Under the condition (H1), the condition (H2) holds in particular if the variable $y$ has a finite moment of order $4$. This seems a fair assumption in practice, unless the variable of interest is heavily skewed like in wealth surveys, for example. \\

\noindent Assumption (H3) is somewhat technical, and is used to ensure that the Horvitz-Thompson estimator is not close to being degenerate. This assumption is only needed to prove a central-limit theorem for the Horvitz-Thompson estimator. It requires that the dispersion within the microstrata does not vanish. For example, it does not hold if $y_k$ is proportional to $\pi_k$.

    \section{Chromy sampling} \label{sec3}

    \noindent \cite{chr:79} proposed a sampling algorithm which is strictly sequential, in the sense that the units in the population are successively considered for possible selection, and the decision for the unit is made at once. The method is presented in Algorithm \ref{chr:samp}. Let us denote
        \begin{eqnarray} \label{sec3:eq1}
          S_c & \sim & Chr(\pi;U)
        \end{eqnarray}
    for a sample selected by means of Chromy sampling with parameter $\pi$ in the population $U$. The method was originally proposed for $\pi$-ps sampling where the inclusion probabilities are defined proportionally on some auxiliary variable, but it is applicable to any set of inclusion probabilities. \\

    \begin{algorithm}%[htb!]
    \begin{enumerate}
    \item At step $k=1$, take $I_1=1$ with probability $\pi_1$.
    \item At step $k=2,\ldots,N$:
        \begin{enumerate}
          \item If $V_k^F > V_{k-1}^F$, then we take the transition probabilities
            \begin{eqnarray*}
              Pr\left(I_k=1 \left| \sum_{l=1}^{k-1} I_l = V_{k-1}^I \right.\right) & = & \frac{V_k^F-V_{k-1}^F}{1-V_{k-1}^F}, \\
              Pr\left(I_k=1 \left| \sum_{l=1}^{k-1} I_l = V_{k-1}^I+1 \right.\right) & = & 0.
            \end{eqnarray*}
          \item If $V_{k-1}^F \geq V_{k}^F$, then we take the transition probabilities
            \begin{eqnarray*}
              Pr\left(I_k=1 \left| \sum_{l=1}^{k-1} I_l = V_{k-1}^I \right.\right) & = & 1, \\
              Pr\left(I_k=1 \left| \sum_{l=1}^{k-1} I_l = V_{k-1}^I+1 \right.\right) & = & \frac{V_k^F}{V_{k-1}^F}.
            \end{eqnarray*}
        \end{enumerate}
    \end{enumerate}
    \caption{Chromy sampling with parameter $\pi$ in the population $U$} \label{chr:samp}
    \end{algorithm}

   \noindent The method proceeds by considering at each step $k=1,\ldots,N$ the unit $k$ for possible selection, and by computing its probability of selection conditionally on the number of units already selected. This algorithm defines a fixed-size sampling design, and \cite{chr:79} proves that the parameter $\pi$ defining the inclusion probabilities is exactly matched. Note that the case $V_k^F > V_{k-1}^F$ (Step 2.a) corresponds to the treatment of a non cross-border unit, while the case $V_{k-1}^F \geq V_k^F$ (Step 2.b) corresponds to the treatment of a cross-border unit. For illustration, the complete probability tree for Chromy sampling on a small population is given in Appendix  \ref{appA}. \\

   \noindent This algorithm allocates the sample regularly in the population, as stated in Proposition \ref{prop1}. The proof is given in Appendix \ref{appB}. At any step $k$ of the procedure, the number of units selected is equal to the sum of inclusion probabilities up to rounding, a property which is sometimes coined as spatial balancing \citep{gra:lun:sch:12}.

    \begin{prop} \label{prop1}
      For any $k=1,\ldots,N$, we have
        \begin{eqnarray} \label{prop1:eq1}
          V_k^I & \leq \sum_{l=1}^k I_l & \leq V_k^I+1.
        \end{eqnarray}
    \end{prop}

    A drawback of the method is that, by construction, two non cross-border units inside the same microstratum $U_i$ may not be selected jointly in the sample. Therefore, many second-order inclusion probabilities are equal to $0$ and the variance of the Horvitz-Thompson estimator may not be unbiasedly estimated. \\

    \noindent For this reason, \cite{chr:79} proposed to use a randomized procedure, which is as follows. The population $U$ is viewed as a closed loop, and we consider the set $\Sigma_c$ of the $N$ possible circular permutations, each of which using a different unit as the first one. The first permutation is $\sigma_1$, with the natural order $1,\ldots,N$. For $k=2,\ldots,N$ the $k$-th permutation is $\sigma_k$ where the units are in the order $k,\ldots,N,1,\ldots,k-1$. The randomized Chromy sample $S_{rc}$ is selected as
        \begin{eqnarray} \label{sec3:eq2}
          S_{rc} & \sim & Chr(\pi^{\sigma_k};U^{\sigma_k}),
        \end{eqnarray}
    with $\sigma_k$ a random permutation selected in $\Sigma_c$ with probability $\pi_k/n$, with $U^{\sigma_k}$ the population ordered with respect to $\sigma_k$, and $\pi^{\sigma_k}$ the vector of probabilities ordered accordingly. This is the algorithm currently implemented in the \verb"SURVEYSELECT" procedure of the \verb"SAS" software.

    \section{Ordered Pivotal sampling} \label{sec4}

    \noindent Ordered pivotal sampling \citep{ful:70,dev:til:98,cha:12} is presented in Algorithm \ref{piv:samp}. This is a succession of duels between units, and at each step the two first units remaining in the population are considered. If the sum of their probabilities is lower than $1$ (rejection step), one of the unit is randomly discarded while the other gets the sum of their probabilities. If the sum of their probabilities is greater than $1$ (selection step), one of the unit is randomly selected while the other goes on with the residual probability. For illustration, the complete probability tree for ordered pivotal sampling on a small population is given in Appendix \ref{appA}. \\

    \begin{algorithm}%[htb!]
    \begin{enumerate}
    \item Initialize with $\pi(0)=\pi_N$.
    \item At step $t=1,\ldots,T$:
        \begin{enumerate}
        \item Initialize with $\pi(t)=\pi(t-1)$.
        \item Take $k<l$ the two first units in the population such that
            \begin{eqnarray*}
              \pi_{k}(t-1) \notin \{0,1\} & \textrm{and} & \pi_{l}(t-1) \notin \{0,1\}
            \end{eqnarray*}
        \item If $\pi_k(t-1)+\pi_l(t-1) \leq 1$ (rejection step), then do:
            \begin{eqnarray*}
              \{\pi_k(t),\pi_l(t)\} & = & \left\{\begin{array}{ll}
                                                 \{\pi_k(t-1)+\pi_l(t-1),0\} & \textrm{with prob. } p(t) \\
                                                 \{0,\pi_k(t-1)+\pi_l(t-1)\} & \textrm{with prob. } 1-p(t),
                                               \end{array}
               \right. \\
               \textrm{where } p(t) & = & \frac{\pi_k(t-1)}{\pi_k(t-1)+\pi_l(t-1)}.
            \end{eqnarray*}
        \item If $\pi_k(t-1)+\pi_l(t-1) > 1$ (selection step), then do:
            \begin{eqnarray*}
              \{\pi_k(t),\pi_l(t)\} & = & \left\{\begin{array}{ll}
                                                 \{1,\pi_k(t-1)+\pi_l(t-1)-1\} & \textrm{with prob. } p(t) \\
                                                 \{\pi_k(t-1)+\pi_l(t-1)-1,1\} & \textrm{with prob. } 1-p(t)
                                               \end{array}
               \right. \\
               \textrm{where } p(t) & = & \frac{1-\pi_l(t-1)}{2-\pi_k(t-1)-\pi_l(t-1)}.
            \end{eqnarray*}
        \end{enumerate}
    \item The algorithm stops at step $T$ when all the components of $\pi(T)$ are $0$ or $1$. Take $I=\pi(T)$.
    \end{enumerate}
    \caption{Ordered pivotal sampling with parameter $\pi$ in the population $U$} \label{piv:samp}
    \end{algorithm}

    \noindent The pivotal sample is selected in at most $N-1$ steps. Pivotal sampling is a particular case of the cube method \citep{dev:til:04}, which enables to perform balanced sampling, i.e. to select samples such that the Horvitz-Thompson estimator exactly matches the known totals for some auxiliary variables. Pivotal sampling has found uses in spatial sampling, since it enables to spread well the sample over space: see for example \cite{gra:lun:sch:12} for the so-called local pivotal method, \cite{cha:leg:18} for the so-called pivotal tesselation method, or \cite{ben:pie:pos:17} for a recent review on spatial sampling methods. Pivotal sampling is also of use in Monte Carlo methods \citep{ger:cho:whi:19}.

\section{Properties of Chromy sampling} \label{sec5}

\noindent We first prove in Theorem \ref{theo1} that Chromy sampling and ordered pivotal sampling are equivalent, which is the main result of the paper. The proof is lengthy, and given in Appendix \ref{appC}.

\begin{theo} \label{theo1}
Ordered pivotal sampling and Chromy sampling with the same parameter $\pi$ induce the same sampling design.
\end{theo}

\noindent By using the characterization of Chromy sampling given in Theorem \ref{theo1}, the mean-square consistency of the Horvitz-Thompson estimator stated in equation (\ref{theo2:eq1}) of Theorem \ref{theo2} is a direct consequence of Theorem 2 in \cite{cha:17}. The central-limit theorem stated in equation (\ref{theo2:eq2}) is a direct consequence of Theorem 1 in \cite{cha:leg:18}. \\

\begin{theo} \label{theo2}
  Suppose that the sample $S_c$ is selected by means of $Chr(\pi;U)$. If assumption (H2) holds, then
    \begin{eqnarray}
      E\left\{N^{-1}(\hat{t}_{y\pi}-t_y)\right\} & = & O(n^{-1}). \label{theo2:eq1}
    \end{eqnarray}
  If in addition assumptions (H1) and (H3) hold, then
    \begin{eqnarray}
      \frac{\hat{t}_{y\pi}-t_y}{\sqrt{V(\hat{t}_{y\pi})}} & \underset{\mathcal{L}}{\longrightarrow} & \mathcal{N}(0,1), \label{theo2:eq2}
    \end{eqnarray}
  where $\underset{\mathcal{L}}{\longrightarrow}$ stands for the convergence in distribution.
\end{theo}

\noindent It also follows from Theorem \ref{theo1} that Chromy sampling is a negatively associated sampling design \citep{joa:pro:83}. This implies that the Sen-Yates-Grundy conditions are satisfied. From Theorem 2 in \cite{ber:cle:19}, the Horvitz-Thompson also satisfies a Bennett/Bernstein-type exponential inequality. \\

\noindent From Theorem 5.1 in \cite{cha:12}, and from the computation given in \cite{dev:98}, it is possible to give an explicit expression for the second-order inclusion probabilities under Chromy sampling. This is the purpose of Theorem \ref{theo3}. \\

\begin{theo} \label{theo3}
    Let $k$ and $l$ be two distinct units in $U$. If $k$ and $l$ are two non cross-border units that belong to the same microstratum
      $U_i$, then
        $$ \pi_{kl}=0,$$
      if $k$ and $l$ are two non cross-border units that belong to distinct microstrata $U_i$ and $U_j$, respectively, where $i<j$, then
        $$ \pi_{kl}=\pi_k \pi_l \left\{ 1-c(i,j)\right\},$$
      if $k=k_{i-1}$ and $l$ is a non cross-border unit that belongs to the microstratum $U_j$ where $i \leq j$, then
        $$ \pi_{kl}=\pi_k \pi_l \left[ 1-b_{i-1} (1-\pi_k) \left\{\pi_k (1-b_{i-1}) \right\}^{-1} c(i,j) \right],$$
      if $l=k_{j-1}$ and $k$ is a non cross-border unit that belongs to the microstratum $U_i$ where $i < j$, then
        $$ \pi_{kl}=\pi_k \pi_l \left\{ 1-(1-\pi_l)(1-b_{j-1}) (\pi_l b_{j-1})^{-1} c(i,j) \right\},$$
      if $k=p_{i-1}$ and $l=p_{j-1}$, where $i < j$, then
        $$ \pi_{kl}=\pi_k \pi_l \left[ 1- b_{i-1} (1-b_{j-1}) (1-\pi_k) (1-\pi_l) \left\{\pi_k \pi_l b_{j-1} (1-b_{i-1}) \right\}^{-1} c(i,j) \right],$$
    where $c(i,j)=\prod_{l=i}^{j-1} c_l$, $c_l=a_l b_l \left\{(1-a_l)(1-b_l)\right\}^{-1}$ and with $c(i,i)=1$.
\end{theo}

\noindent It is clear from Theorem \ref{theo3} that many second-order inclusion probabilities are equal to zero for Chromy sampling. It also makes possible to compute the second-order inclusion probabilities for randomized Chromy sampling. For any units $k \neq l \in U$, let us denote $\pi_{kl}^{rc}$ their second-order inclusion probability under randomized Chromy sampling. Then:
    \begin{eqnarray} \label{sec5:eq1}
    \pi_{kl}^{rc} & = & \sum_{i \in U} \frac{\pi_i}{n} \pi_{kl}^{\sigma_i},
    \end{eqnarray}
with $\pi_{kl}^{\sigma_i}$ the joint selection probabilities of units $k$ and $l$ with the permutation $\sigma_i$, i.e. when Chromy sampling is applied to the population $U^{\sigma_i}$ with parameter $\pi^{\sigma_i}$. A \verb"SAS IML" subroutine to compute the second-order inclusion probabilities in $(\ref{sec5:eq1})$ is available as  Supplementary Material. \\

\noindent We evaluate in the simulation study performed in Section \ref{sec6} a variance estimator making use of second-order inclusion probabilities computed from equation (\ref{sec5:eq1}). For illustration, we give in this Section a small example. We consider a population $U$ of size $N=8$, with the parameter $$\pi=(0.2,0.4,0.7,0.4,0.6,0.6,0.3,0.8)^{\top},$$ which leads to a sample of size $n=4$. From equation (\ref{sec5:eq1}), we obtain the following matrix of second-order inclusion probabilities (rounded to three decimal places):
    \begin{eqnarray*}
      (\pi_{kl}^{rc}) & = &
      \left(
      \begin{array}{cccccccc}
    0.200 & 0.041 & 0.133 & 0.075 & 0.116 & 0.108 & 0.046 & 0.081 \\
          & 0.400 & 0.171 & 0.142 & 0.224 & 0.227 & 0.099 & 0.297 \\
          &       & 0.700 & 0.209 & 0.410 & 0.415 & 0.207 & 0.555 \\
          &       &       & 0.400 & 0.118 & 0.224 & 0.113 & 0.319 \\
          &       &       &       & 0.600 & 0.293 & 0.165 & 0.474 \\
          &       &       &       &       & 0.600 & 0.065 & 0.469 \\
          &       &       &       &       &       & 0.300 & 0.205 \\
          &       &       &       &       &       &       & 0.800 \\
      \end{array}
      \right).
    \end{eqnarray*}
We also selected $10^6$ samples by means of the \verb"SURVEYSELECT" procedure with the option \verb"METHOD=PPS_SEQ", which leads to randomized Chromy sampling. These $10^6$ samples are used to obtain a simulation-based approximation of the matrix of second-order inclusion probabilities, which is given below (rounded to three decimal places):
    \begin{eqnarray*}
      (\pi_{kl,sim}^{rc}) & = &
      \left(
      \begin{array}{cccccccc}
        0.200 & 0.041 & 0.133 & 0.075 & 0.116 & 0.108 & 0.046 & 0.081 \\
              & 0.400 & 0.171 & 0.142 & 0.223 & 0.227 & 0.099 & 0.296 \\
              &       & 0.701 & 0.210 & 0.410 & 0.416 & 0.208 & 0.556 \\
              &       &       & 0.400 & 0.118 & 0.225 & 0.113 & 0.318 \\
              &       &       &       & 0.600 & 0.293 & 0.165 & 0.474 \\
              &       &       &       &       & 0.600 & 0.065 & 0.469 \\
              &       &       &       &       &       & 0.300 & 0.205 \\
              &       &       &       &       &       &       & 0.800 \\
          \end{array}
          \right).
    \end{eqnarray*}
It is clear that both matrices are almost identical.

\section{Simulation study} \label{sec6}

\noindent We conducted a simulation study in order to evaluate variance estimation and interval estimation for randomized Chromy sampling. The set-up is inspired from \cite{cha:haz:les:17}. We generate $2$ populations of size $N=500$, each consisting of an auxiliary variable $x$ and $4$ variables of interest $y_1,\ldots,y_4$. In the first population, the $x$-values are generated according to a Gamma distribution with shape and scale parameters $2$ and $2$; in the second population, the $x$-values are generated from a log-normal distribution with parameters $0$ and $1.7$. The $x$-values are then shaped and scaled to lie between $1$ and $10$. \\

\noindent  Given the $x$-values, the values of the variables of interest are generated according to the following models:
    \begin{eqnarray} \label{sec6:eq1}
      \verb"linear": y_{1k} & = & \alpha_{10} + \alpha_{11} (x_k - \mu_x) + \sigma_1~\epsilon_k, \nonumber \\
      \verb"quadratic": y_{2k} & = & \alpha_{20} + \alpha_{21} (x_k - \mu_x)^2 + \sigma_2~\epsilon_k, \\
      \verb"exponential": y_{3k} & = & \exp\{\alpha_{30} + \alpha_{31}(x_k - \mu_x)\}+\sigma_3~\epsilon_k, \nonumber \\
      \verb"bump": y_{4k} & = & \alpha_{40} + \alpha_{41} (x_k - \mu_x)^2 - \alpha_{42} \exp\left\lbrace-\alpha_{43}( x_k - \mu_x)^2\right\rbrace + \sigma_4~\epsilon_k, \nonumber
    \end{eqnarray}
where $\mu_x$ is the population mean of $x$, and where $\epsilon_k$ follows a standard normal distribution $\mathcal{N}(0,1)$. The population mean $\mu_y$ and the population dispersion $S_y^2$ for the two populations and the four variables of interest are given in Table \ref{sim:tab1}.

\verbdef{\linear}{linear}
\verbdef{\quadratic}{quadratic}
\verbdef{\exponential}{exponential}
\verbdef{\bump}{bump}

\begin{table}[h!]
\caption{Population mean and population dispersion for two populations and four variables of interest} \label{sim:tab1}
	\begin{center}
	\begin{tabular}{|l|cc|cc|cc|cc|} \hline
                   & \multicolumn{2}{|c|}{\linear} & \multicolumn{2}{|c|}{\quadratic} & \multicolumn{2}{|c|}{\exponential} & \multicolumn{2}{|c|}{\bump}  \\
                   & $\mu_{y1}$ & $S_{y1}^2$ & $\mu_{y2}$ & $S_{y2}^2$ & $\mu_{y3}$ & $S_{y3}^2$ & $\mu_{y4}$ & $S_{y4}^2$ \\ \hline \hline
    Population $1$ & $10.1$ & $13.1$ & $11.7$ & $69.8$ & $10.3$ & $18.2$ & $12.3$ & $84.6$ \\
    Population $2$ & $10.1$ & $5.3$  & $8.8$ & $8.3$ & $9.9$ & $4.4$ & $5.2$ & $18.0$ \\ \hline
	\end{tabular}
	\end{center}
\end{table}

\noindent In each population, we computed inclusion probabilities proportional to the $x$-values, according to the formula
    \begin{eqnarray} \label{sec6:eq2}
      \pi_k & = & n \frac{x_k}{\sum_{l \in U} x_l},
    \end{eqnarray}
with $n=50,100$ or $200$. The range of inclusion probabilities is given in Table \ref{sim:tab2}. In some cases, equation (\ref{sec6:eq2}) leads to inclusion probabilities greater than $1$ for some units. In such case, the corresponding units are selected with certainty ($\pi_k=1$), and the other probabilities are recomputed. For the first population, $5$ units are selected with certainty with $n=200$. For the second population, $2$ units are selected with certainty with $n=100$ and $9$ units are selected with certainty with $n=100$. \\

\begin{table}[h!]
\caption{Range of inclusion probabilities proportional to $x$ for two populations and three sample sizes} \label{sim:tab2}
	\begin{center}
	\begin{tabular}{|l|cc|cc|cc|} \hline
                   & \multicolumn{2}{|c|}{$n=50$} & \multicolumn{2}{|c|}{$n=100$} & \multicolumn{2}{|c|}{$n=200$}  \\
                   & Min   & Max   & Min   & Max   & Min   & Max \\ \hline \hline
    Population $1$ & 0.030 & 0.305 & 0.060 & 0.609 & 0.122 & 1.000 \\
    Population $2$ & 0.079 & 0.791 & 0.160 & 1.000 & 0.325 & 1.000 \\ \hline
	\end{tabular}
	\end{center}
\end{table}

\noindent We select $B=1,000$ samples by means of randomized Chromy sampling, using the \verb"SURVEYSELECT" procedure. For each sample and each variable of interest, we compute the Horvitz-Thompson estimator $\hat{t}_{y\pi}$, and the Sen-Yates-Grundy variance estimator
    \begin{eqnarray}
      \hat{V}(\hat{t}_{y\pi}) & = & \frac{1}{2} \sum_{k \neq l \in S_{rc}} \frac{\pi_k \pi_l - \pi_{kl}^{rc}}{\pi_{kl}^{rc}} \left(\frac{y_k}{\pi_k}-\frac{y_l}{\pi_l} \right)^2,
    \end{eqnarray}
where the second-order inclusion probabilities are given by equation (\ref{sec5:eq1}). To evaluate the properties of this variance estimator, we compute the relative bias
    \begin{eqnarray} \label{rel:bias}
      RB\{\hat{V}(\hat{t}_{y\pi})\} & = & 100 \times \frac{B^{-1} \sum_{b=1}^B \hat{V}_b(\hat{t}_{y\pi b})-V(\hat{t}_{y\pi})}{V(\hat{t}_{y\pi})},
    \end{eqnarray}
where $\hat{V}_b(\hat{t}_{y\pi b})$ denotes the variance estimator in the b-th sample, and where $V(\hat{t}_{y\pi})$ is the exact variance, computed by using the exact second-order inclusion probabilities given in (\ref{sec5:eq1}). As a measure of stability, we use the Relative Root Mean Square Error
    \begin{eqnarray*} \label{instab}
      RRMSE\{\hat{V}(\hat{t}_{y\pi})\} & = & 100 \times \frac{\left[B^{-1}\sum_{b=1}^B \left\{\hat{V}_b(\hat{t}_{y\pi b})-V(\hat{t}_{y\pi}) \right\}^2\right]^{1/2}}{V(\hat{t}_{y\pi})}.
    \end{eqnarray*}
Finally, we compute the error rate of the normality-based confidence intervals with nominal one-tailed error rate of 2.5 \% in each tail. \\

\noindent The simulation results are given in Table \ref{sim:tab3}. The Sen-Yates-Grundy variance estimator is almost unbiased in all cases considered, except for the first population with $n=50$ where the variance estimator is slightly positively biased. The Relative Root Mean Square Error diminishes as $n$ increases, as expected. We note that the coverage rates are not well respected with $n=50$, which is likely due to the small sample size and to the instability of the variance estimator. When the sample size increases, the coverage rates become close to the nominal level.

\begin{sidewaystable}%[htb!]
\caption{Relative bias (in \% ), Relative Root Mean Square Error (in \% ) of the variance estimator and Coverage Rate of the normality-based confidence intervals for three populations and four variables of interest} \label{sim:tab3}
	\begin{center}
	\begin{tabular}{|l|l|ccc|ccc|ccc|} \hline
                   & & \multicolumn{3}{|c|}{$n=50$} & \multicolumn{3}{c|}{$n=100$} & \multicolumn{3}{c|}{$n=200$}  \\
                   & & RB & RMSE & Cov. Rate & RB & RMSE & Cov. Rate & RB & RMSE & Cov. Rate \\ \hline
                   & \verb"linear"      & -0.8 & 76  & 11.20 & -0.5 & 46 & 7.20 & -1.0 & 28 & 5.50 \\
    Population $1$ & \verb"quadratic"   & 3.2  & 118 & 13.50 & -0.6 & 66 & 8.60 & -2.0 & 37 & 7.00 \\
                   & \verb"exponential" & -0.4 & 85  & 11.40 & -0.9 & 50 & 7.10 & -1.1 & 30 & 6.50 \\
                   & \verb"bump"        & 3.7  & 103 & 11.30 & -0.4 & 59 & 8.30 & -2.1 & 32 & 7.00 \\ \hline
                   & \verb"linear"      & -0.7 & 56  & 8.20  & 0.6  & 31 & 7.30 & 0.0  & 17 & 4.80 \\
    Population $2$ & \verb"quadratic"   & -0.7 & 56  & 8.60  & 0.6  & 30 & 7.70 & 0.0  & 17 & 4.80 \\
                   & \verb"exponential" & -0.6 & 55  & 8.10  & 0.5  & 30 & 7.50 & 0.0  & 17 & 4.90 \\
                   & \verb"bump"        & -1.4 & 61  & 11.10 & 0.8  & 35 & 6.90 & 0.9  & 19 & 4.30 \\ \hline
	\end{tabular}
	\end{center}
\end{sidewaystable}

\section{Conclusion} \label{sec7}

\noindent In this paper, we have studied Chromy's sampling algorithm. We proved that it is equivalent to ordered pivotal sampling, which enables in particular computing the second-order inclusion probabilities for the randomized Chromy algorithm programmed in the \verb"SURVEYSELECT" procedure. The results in our simulation study confirm that the variance estimator based on the second-order probabilities computed from Deville's formulas show almost no bias for moderate sample sizes. \\

\noindent The number of computations for the second-order inclusion probabilities of randomized Chromy sampling is of order $N^3$. Formula (\ref{sec5:eq1}) is therefore tractable in case of a small population, for example when Chromy sampling is used to select a set of Primary Sampling Units in a multistage survey \citep[e.g.][]{rus:mye:dan:que:las:rib:19}. Otherwise, we may resort to a simulation-based approximation of equation (\ref{sec5:eq1}).

\bibliographystyle{apalike}
%\bibliography{C:/Users/ENSAI/Dropbox/RedactionArticles/_Biblio/biblio_gen}

\newpage
\appendix

\section{Chromy sampling and pivotal sampling on an example} \label{appA}

\noindent We consider the population $U=\{1,2,3,4,5\}$ with $\pi=(0.4,0.8,0.5,0.6,0.7)$. The complete probability tree for Chromy sampling with parameter $\pi$ is given in Figure \ref{prob:tree:2}. For example, at the first step the unit $1$ is selected with probability $0.4$, and discarded with probability $0.6$. If unit $1$ is selected, then $\sum_{l=1}^1 I_l = V_1^I+1$. Since unit $2$ is a cross-border unit, we follow Step 2.b of Algorithm \ref{chr:samp} and unit $2$ is selected at the next step with probability $V_2^F/V_1^F = 0.2/0.4=1/2$. \\

\noindent The complete sampling design is
    \begin{eqnarray}
      S_c & = & \left\{ \begin{array}{lll}
                          \{1,2,4\} & \textrm{ with proba. } & 3/35, \\
                          \{1,2,5\} & \textrm{ with proba. } & 4/35, \\
                          \{1,3,4\} & \textrm{ with proba. } & 3/56, \\
                          \{1,3,5\} & \textrm{ with proba. } & 1/14, \\
                          \{1,4,5\} & \textrm{ with proba. } & 3/40, \\
                          \{2,3,4\} & \textrm{ with proba. } & 9/56, \\
                          \{2,3,5\} & \textrm{ with proba. } & 3/14, \\
                          \{2,4,5\} & \textrm{ with proba. } & 9/40.
                        \end{array}
      \right.
    \end{eqnarray}

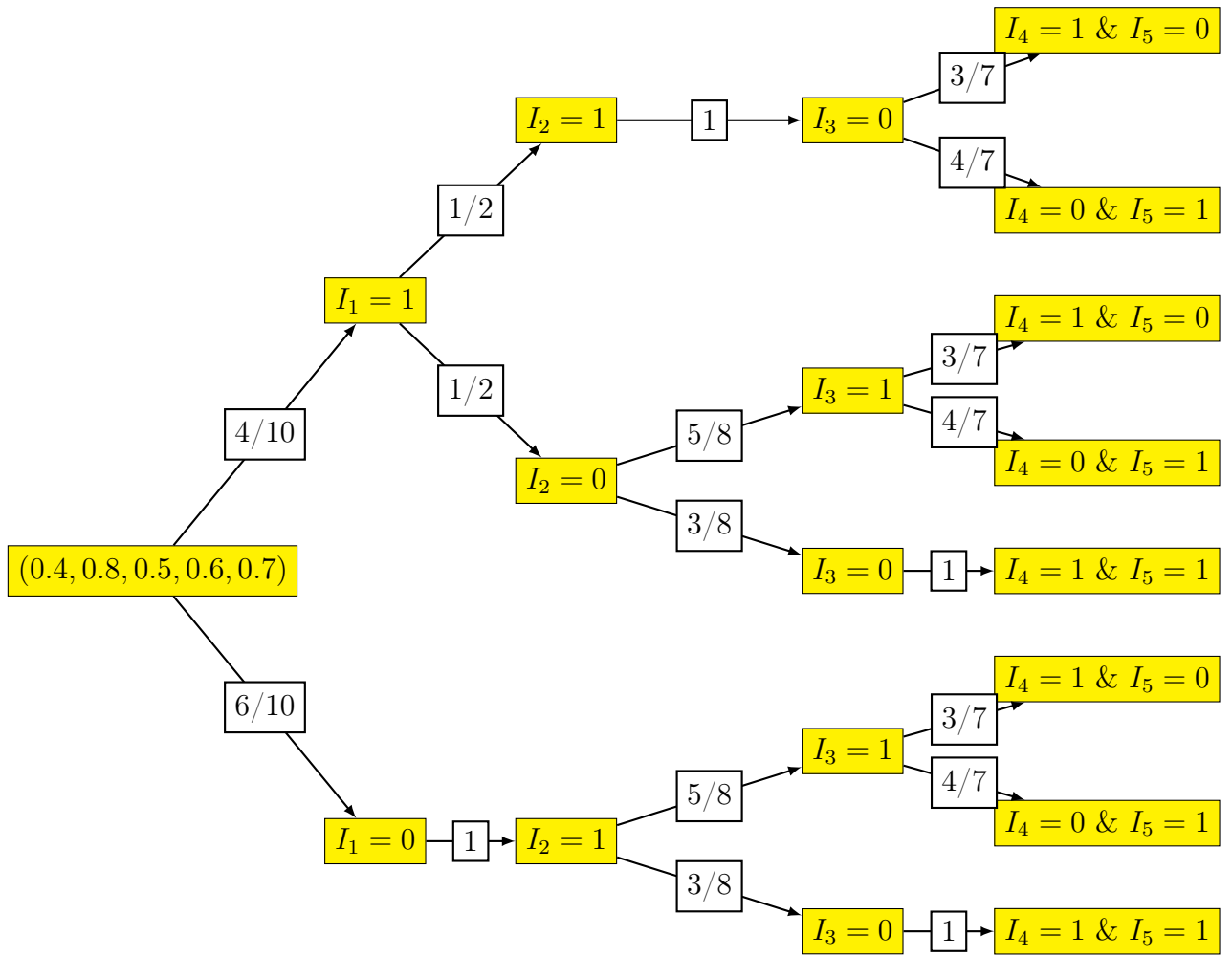
\begin{figure}
%:-+-+-+- Engendré par : http://math.et.info.free.fr/TikZ/Arbre/
\begin{center}
% Racine à Gauche, développement vers la droite
\begin{tikzpicture}[xscale=0.5,yscale=0.5]
% Styles (MODIFIABLES)
\tikzstyle{fleche}=[->,>=latex,thick]
\tikzstyle{noeud}=[fill=yellow,rectangle,draw]
\tikzstyle{feuille}=[fill=yellow,rectangle,draw]
\tikzstyle{etiquette}=[midway,fill=white,draw]
% Dimensions (MODIFIABLES)
\def\DistanceInterNiveaux{3.5}
\def\DistanceInterFeuilles{2.5}
% Dimensions calculées (NON MODIFIABLES)
\def\NiveauA{(0)*\DistanceInterNiveaux}
\def\NiveauB{(1.75)*\DistanceInterNiveaux}
\def\NiveauC{(3.25)*\DistanceInterNiveaux}
\def\NiveauD{(5.5)*\DistanceInterNiveaux}
\def\NiveauE{(7.5)*\DistanceInterNiveaux}
\def\InterFeuilles{(-1)*\DistanceInterFeuilles}
% Noeuds (MODIFIABLES : Styles et Coefficients d'InterFeuilles)
\node[noeud] (R) at ({\NiveauA},{(6.5)*\InterFeuilles}) {$(0.4,0.8,0.5,0.6,0.7)$};
\node[noeud] (Ra) at ({\NiveauB},{(3.5)*\InterFeuilles}) {$I_1=1$};
\node[noeud] (Raa) at ({\NiveauC},{(1.5)*\InterFeuilles}) {$I_2=1$};
\node[noeud] (Raaa) at ({\NiveauD},{(1.5)*\InterFeuilles}) {$I_3=0$};
\node[feuille] (Raaaa) at ({\NiveauE},{(0.5)*\InterFeuilles}) {$I_4=1~\&~I_5=0$};
\node[feuille] (Raaab) at ({\NiveauE},{(2.5)*\InterFeuilles}) {$I_4=0~\&~I_5=1$};
\node[noeud] (Rab) at ({\NiveauC},{(5.5)*\InterFeuilles}) {$I_2=0$};
\node[noeud] (Raba) at ({\NiveauD},{(4.5)*\InterFeuilles}) {$I_3=1$};
\node[feuille] (Rabaa) at ({\NiveauE},{(3.7)*\InterFeuilles}) {$I_4=1~\&~I_5=0$};
\node[feuille] (Rabab) at ({\NiveauE},{(5.3)*\InterFeuilles}) {$I_4=0~\&~I_5=1$};
\node[noeud] (Rabb) at ({\NiveauD},{(6.5)*\InterFeuilles}) {$I_3=0$};
\node[feuille] (Rabba) at ({\NiveauE},{(6.5)*\InterFeuilles}) {$I_4=1~\&~I_5=1$};
\node[noeud] (Rb) at ({\NiveauB},{(9.5)*\InterFeuilles}) {$I_1=0$};
\node[noeud] (Rba) at ({\NiveauC},{(9.5)*\InterFeuilles}) {$I_2=1$};
\node[noeud] (Rbaa) at ({\NiveauD},{(8.5)*\InterFeuilles}) {$I_3=1$};
\node[feuille] (Rbaaa) at ({\NiveauE},{(7.7)*\InterFeuilles}) {$I_4=1~\&~I_5=0$};
\node[feuille] (Rbaab) at ({\NiveauE},{(9.3)*\InterFeuilles}) {$I_4=0~\&~I_5=1$};
\node[noeud] (Rbab) at ({\NiveauD},{(10.5)*\InterFeuilles}) {$I_3=0$};
\node[feuille] (Rbaba) at ({\NiveauE},{(10.5)*\InterFeuilles}) {$I_4=1~\&~I_5=1$};
% Arcs (MODIFIABLES : Styles)
\draw[fleche] (R)--(Ra) node[etiquette] {4/10};
\draw[fleche] (Ra)--(Raa) node[etiquette] {1/2};
\draw[fleche] (Raa)--(Raaa) node[etiquette] {1};
\draw[fleche] (Raaa)--(Raaaa) node[etiquette] {3/7};
\draw[fleche] (Raaa)--(Raaab) node[etiquette] {4/7};
\draw[fleche] (Ra)--(Rab) node[etiquette] {1/2};
\draw[fleche] (Rab)--(Raba) node[etiquette] {5/8};
\draw[fleche] (Raba)--(Rabaa) node[etiquette] {3/7};
\draw[fleche] (Raba)--(Rabab) node[etiquette] {4/7};
\draw[fleche] (Rab)--(Rabb) node[etiquette] {3/8};
\draw[fleche] (Rabb)--(Rabba) node[etiquette] {1};
\draw[fleche] (R)--(Rb) node[etiquette] {6/10};
\draw[fleche] (Rb)--(Rba) node[etiquette] {1};
\draw[fleche] (Rba)--(Rbaa) node[etiquette] {5/8};
\draw[fleche] (Rbaa)--(Rbaaa) node[etiquette] {3/7};
\draw[fleche] (Rbaa)--(Rbaab) node[etiquette] {4/7};
\draw[fleche] (Rba)--(Rbab) node[etiquette] {3/8};
\draw[fleche] (Rbab)--(Rbaba) node[etiquette] {1};
\end{tikzpicture}
\end{center}
%:-+-+-+-+- Fin
\caption{Probability tree for Chromy sampling on a population $U$ of size $N=5$} \label{prob:tree:2}
\end{figure}

\noindent Now, we consider ordered pivotal sampling on the same population $U$ with the same parameter $\pi$. The complete probability tree for pivotal sampling is given in Figure \ref{prob:tree:1}. For example, at the first step, the units $1$ and $2$ fight with respective probabilities $0.4$ and $0.8$. With probability $(1-0.8)/(2-0.4-0.8)=1/4$, unit $1$ is selected and unit $2$ gets the residual probability $0.2$, and with the complementary probability unit $2$ is selected and unit $1$ gets the residual probability $0.2$. In the first case, unit $2$ faces unit $3$ with respective probabilities $0.2$ and $0.5$. With probability $0.2/(0.2+0.5)=2/7$, unit $2$ gets the sum of the probabilities and unit $3$ is discarded, and with the complementary probability unit $3$ gets the sum of the probabilities and unit $2$ is discarded. \\

\noindent It follows from straightforward computations that the complete sampling design is the same as for Chromy sampling.

\begin{figure}
%:-+-+-+- Engendré par : http://math.et.info.free.fr/TikZ/Arbre/
\begin{center}
% Racine à Gauche, développement vers la droite
\begin{tikzpicture}[xscale=0.5,yscale=0.5]
% Styles (MODIFIABLES)
\tikzstyle{fleche}=[->,>=latex,thick]
\tikzstyle{noeud}=[fill=yellow,rectangle,draw]
\tikzstyle{feuille}=[fill=yellow,rectangle,draw]
\tikzstyle{etiquette}=[midway,fill=white,draw]
% Dimensions (MODIFIABLES)
\def\DistanceInterNiveaux{3.5}
\def\DistanceInterFeuilles{2.5}
% Dimensions calculées (NON MODIFIABLES)
\def\NiveauA{(0)*\DistanceInterNiveaux}
\def\NiveauB{(1.75)*\DistanceInterNiveaux}
\def\NiveauC{(3.25)*\DistanceInterNiveaux}
\def\NiveauD{(5.5)*\DistanceInterNiveaux}
\def\NiveauE{(7.5)*\DistanceInterNiveaux}
\def\InterFeuilles{(-1)*\DistanceInterFeuilles}
% Noeuds (MODIFIABLES : Styles et Coefficients d'InterFeuilles)
\node[noeud] (R) at ({\NiveauA},{(7.5)*\InterFeuilles}) {$(0.4,0.8,0.5,0.6,0.7)$};
\node[noeud] (Ra) at ({\NiveauB},{(3.5)*\InterFeuilles}) {$(1,0.2,0.5,0.6,0.7)$};
\node[noeud] (Raa) at ({\NiveauC},{(1.5)*\InterFeuilles}) {$(1,0.7,0,0.6,0.7)$};
\node[noeud] (Raaa) at ({\NiveauD},{(0.5)*\InterFeuilles}) {$(1,1,0,0.3,0.7)$};
\node[feuille] (Raaaa) at ({\NiveauE},{(0)*\InterFeuilles}) {$(1,1,0,1,0)$};
\node[feuille] (Raaab) at ({\NiveauE},{(1)*\InterFeuilles}) {$(1,1,0,0,1)$};
\node[noeud] (Raab) at ({\NiveauD},{(2.5)*\InterFeuilles}) {$(1,0.3,0,1,0.7)$};
\node[feuille] (Raaba) at ({\NiveauE},{(2)*\InterFeuilles}) {$(1,1,0,1,0)$};
\node[feuille] (Raabb) at ({\NiveauE},{(3)*\InterFeuilles}) {$(1,0,0,1,1)$};
\node[noeud] (Rab) at ({\NiveauC},{(5.5)*\InterFeuilles}) {$(1,0,0.7,0.6,0.7)$};
\node[noeud] (Raba) at ({\NiveauD},{(4.5)*\InterFeuilles}) {$(1,0,1,0.3,0.7)$};
\node[feuille] (Rabaa) at ({\NiveauE},{(4)*\InterFeuilles}) {$(1,0,1,1,0)$};
\node[feuille] (Rabab) at ({\NiveauE},{(5)*\InterFeuilles}) {$(1,0,1,0,1)$};
\node[noeud] (Rabb) at ({\NiveauD},{(6.5)*\InterFeuilles}) {$(1,0,0.3,1,0.7)$};
\node[feuille] (Rabba) at ({\NiveauE},{(6)*\InterFeuilles}) {$(1,0,1,1,0)$};
\node[feuille] (Rabbb) at ({\NiveauE},{(7)*\InterFeuilles}) {$(1,0,0,1,1)$};
\node[noeud] (Rb) at ({\NiveauB},{(11.5)*\InterFeuilles}) {$(0.2,1,0.5,0.6,0.7)$};
\node[noeud] (Rba) at ({\NiveauC},{(9.5)*\InterFeuilles}) {$(0.7,1,0,0.6,0.7)$};
\node[noeud] (Rbaa) at ({\NiveauD},{(8.5)*\InterFeuilles}) {$(1,1,0,0.3,0.7)$};
\node[feuille] (Rbaaa) at ({\NiveauE},{(8)*\InterFeuilles}) {$(1,1,0,1,0)$};
\node[feuille] (Rbaab) at ({\NiveauE},{(9)*\InterFeuilles}) {$(1,1,0,0,1)$};
\node[noeud] (Rbab) at ({\NiveauD},{(10.5)*\InterFeuilles}) {$(0.3,1,0,1,0.7)$};
\node[feuille] (Rbaba) at ({\NiveauE},{(10)*\InterFeuilles}) {$(1,1,0,1,0)$};
\node[feuille] (Rbabb) at ({\NiveauE},{(11)*\InterFeuilles}) {$(0,1,0,1,1)$};
\node[noeud] (Rbb) at ({\NiveauC},{(13.5)*\InterFeuilles}) {$(0,1,0.7,0.6,0.7)$};
\node[noeud] (Rbba) at ({\NiveauD},{(12.5)*\InterFeuilles}) {$(0,1,1,0.3,0.7)$};
\node[feuille] (Rbbaa) at ({\NiveauE},{(12)*\InterFeuilles}) {$(0,1,1,1,0)$};
\node[feuille] (Rbbab) at ({\NiveauE},{(13)*\InterFeuilles}) {$(0,1,1,0,1)$};
\node[noeud] (Rbbb) at ({\NiveauD},{(14.5)*\InterFeuilles}) {$(0,1,0.3,1,0.7)$};
\node[feuille] (Rbbba) at ({\NiveauE},{(14)*\InterFeuilles}) {$(0,1,1,1,0)$};
\node[feuille] (Rbbbb) at ({\NiveauE},{(15)*\InterFeuilles}) {$(0,1,0,1,1)$};
% Arcs (MODIFIABLES : Styles)
\draw[fleche] (R)--(Ra) node[etiquette] {1/4};
\draw[fleche] (Ra)--(Raa) node[etiquette] {2/7};
\draw[fleche] (Raa)--(Raaa) node[etiquette] {4/7};
\draw[fleche] (Raaa)--(Raaaa) node[etiquette] {3/10};
\draw[fleche] (Raaa)--(Raaab) node[etiquette] {7/10};
\draw[fleche] (Raa)--(Raab) node[etiquette] {3/7};
\draw[fleche] (Raab)--(Raaba) node[etiquette] {3/10};
\draw[fleche] (Raab)--(Raabb) node[etiquette] {7/10};
\draw[fleche] (Ra)--(Rab) node[etiquette] {5/7};
\draw[fleche] (Rab)--(Raba) node[etiquette] {4/7};
\draw[fleche] (Raba)--(Rabaa) node[etiquette] {3/10};
\draw[fleche] (Raba)--(Rabab) node[etiquette] {7/10};
\draw[fleche] (Rab)--(Rabb) node[etiquette] {3/7};
\draw[fleche] (Rabb)--(Rabba) node[etiquette] {3/10};
\draw[fleche] (Rabb)--(Rabbb) node[etiquette] {7/10};
\draw[fleche] (R)--(Rb) node[etiquette] {3/4};
\draw[fleche] (Rb)--(Rba) node[etiquette] {2/7};
\draw[fleche] (Rba)--(Rbaa) node[etiquette] {4/7};
\draw[fleche] (Rbaa)--(Rbaaa) node[etiquette] {3/10};
\draw[fleche] (Rbaa)--(Rbaab) node[etiquette] {7/10};
\draw[fleche] (Rba)--(Rbab) node[etiquette] {3/7};
\draw[fleche] (Rbab)--(Rbaba) node[etiquette] {3/10};
\draw[fleche] (Rbab)--(Rbabb) node[etiquette] {7/10};
\draw[fleche] (Rb)--(Rbb) node[etiquette] {5/7};
\draw[fleche] (Rbb)--(Rbba) node[etiquette] {4/7};
\draw[fleche] (Rbba)--(Rbbaa) node[etiquette] {3/10};
\draw[fleche] (Rbba)--(Rbbab) node[etiquette] {7/10};
\draw[fleche] (Rbb)--(Rbbb) node[etiquette] {3/7};
\draw[fleche] (Rbbb)--(Rbbba) node[etiquette] {3/10};
\draw[fleche] (Rbbb)--(Rbbbb) node[etiquette] {7/10};
\end{tikzpicture}
\end{center}
%:-+-+-+-+- Fin
\caption{Probability tree for ordered pivotal sampling on a clustered population $U_c$} \label{prob:tree:1}
\end{figure}
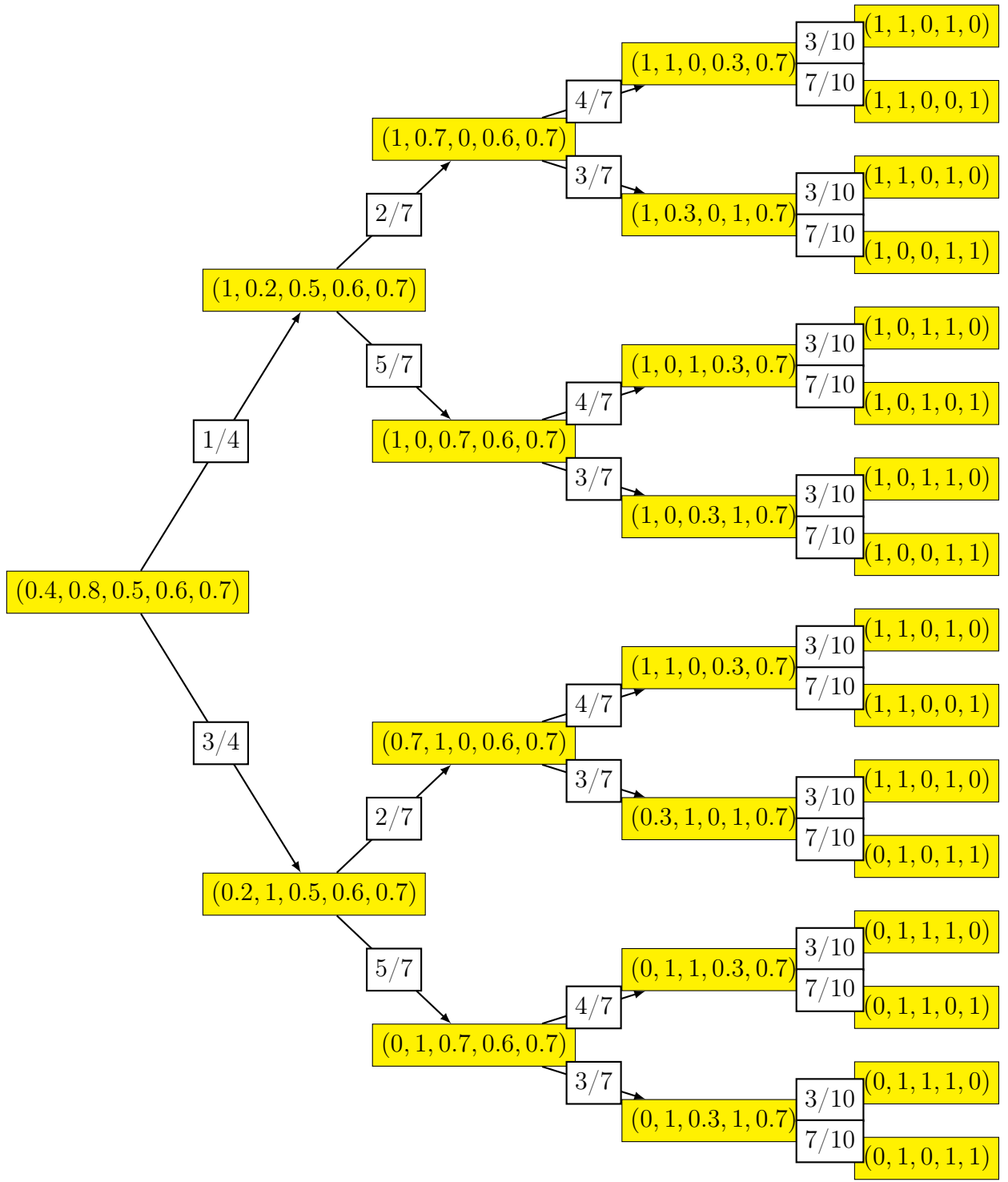

\section{Proof of Proposition \ref{prop1}} \label{appB}

\noindent The proof is by induction. For $k=1$, we have $I_1 \in \{0,1\}$ and $V_1^I=0$, so the property holds. Suppose that the property holds at $k-1$, namely
\begin{eqnarray} \label{pprop1:eq1}
V_{k-1}^I & \leq \sum_{l=1}^{k-1} I_l & \leq V_{k-1}^I+1.
\end{eqnarray}
First note that if $k$ is a cross-border unit, we have $V_{k}^I=V_{k-1}^I+1$ and if not, we have $V_{k}^I=V_{k-1}^I$.
\begin{itemize}
  \item Suppose that $\sum_{l=1}^{k-1} I_l=V_{k-1}^I$. If $k$ is a non cross-border unit, we have
    \begin{eqnarray*}
    V_{k-1}^I \leq \sum_{l=1}^{k-1} I_l + I_k \leq V_{k-1}^I+1 & \Leftrightarrow &
    V_{k}^I \leq \sum_{l=1}^{k} I_l \leq V_{k}^I+1.
    \end{eqnarray*}
  If $k$ is a cross-border unit, we obtain from Algorithm \ref{chr:samp} that $I_k=1$, and $\sum_{l=1}^{k} I_l=V_{k-1}^I+1=V_{k}^I$.
  \item Suppose that $\sum_{l=1}^{k-1} I_l=V_{k-1}^I+1$. If $k$ is a non cross-border unit, then from Algorithm \ref{chr:samp}, we have $I_k=0$ and $\sum_{l=1}^{k} I_l=V_{k-1}^I+1=V_{k}^I$. If $k$ is a cross-border unit, we obtain from $I_k \in \{0,1\}$
    \begin{eqnarray*}
    V_{k-1}^I+1 \leq \sum_{l=1}^{k-1} I_l + I_k \leq V_{k-1}^I+2 & \Leftrightarrow & V_{k}^I \leq \sum_{l=1}^{k} I_l \leq V_{k}^I+1.
    \end{eqnarray*}
\end{itemize}

\newpage

\section{Proof of Theorem 1} \label{appC}

\noindent The proof proceeds in two main steps. We first prove that Chromy sampling may be alternatively seen as the result of a two-stage sampling procedure, inside a population $U_c$ of clusters which is introduced in Section \ref{appC1}. We then consider in Section \ref{appC2} the ordered sample, which is given by the selected units ranked with respect to the natural order in the population, and we give the transition probabilities between the selected units. These results are used in Section \ref{appC3} to prove Theorem \ref{theo1}.

    \subsection{Clustered population} \label{appC1}

    \noindent The $N$ sampling units in the population $U$ are grouped to obtain a population $U_c=\{u_1,\ldots,u_{2n-1}\}$ of clusters. There are the clusters associated to the cross-border units ($n-1$ singletons), denoted as $u_{2i}$ with associated probability $\phi_{2i}=\pi_{k_i}$. There are the $n$ clusters of non cross-border units that are between two consecutive integers, denoted as $u_{2i-1}$ with associated probability $\phi_{2i-1}=V_{k_i-1}-V_{k_{i-1}}$. The vector of inclusion probabilities in the population $U_c$ is denoted as $\phi=\left(\phi_1,\ldots,\phi_{2n-1}\right)'$. For illustration, useful quantities for population $U_c$ are presented in Figure \ref{nota:clust}.

        \begin{center}
        \begin{figure}[htb!]

        \setlength{\unitlength}{0.03cm}

        \begin{picture}(600,200)

        % Le segment principal
        \put(30,130){\line(1,0){350}}

        \put(20,130){\line(1,0){5}}
        \put(10,130){\line(1,0){5}}
        \put(0,130){\line(1,0){5}}

        \put(385,130){\line(1,0){5}}
        \put(395,130){\line(1,0){5}}
        \put(405,130){\line(1,0){5}}

        % Les bornes enti\`eres
        \put(100,120){\line(0,1){20}} \put(100,145){\makebox(0,0)[b]{$i-1$}}
        \put(200,120){\line(0,1){20}} \put(200,145){\makebox(0,0)[b]{$i$}}
        \put(300,120){\line(0,1){20}} \put(300,145){\makebox(0,0)[b]{$i+1$}}

        % Les V(k)
        \put(60,125){\line(0,1){10}}
        \put(120,125){\line(0,1){10}}
        \put(170,125){\line(0,1){10}}
        \put(230,125){\line(0,1){10}}
        \put(280,125){\line(0,1){10}}
        \put(340,125){\line(0,1){10}}

        % Les unit\'es frontalières
        \put(60,85){\vector(1,0){40}}  \put(100,85){\vector(-1,0){40}} \put(80,75){\makebox(0,0)[b]{\tiny{$a_{i-1}$}}}
        \put(100,85){\vector(1,0){20}} \put(120,85){\vector(-1,0){20}} \put(110,75){\makebox(0,0)[b]{\tiny{$b_{i-1}$}}}

        \put(170,85){\vector(1,0){30}}  \put(200,85){\vector(-1,0){30}} \put(185,75){\makebox(0,0)[b]{\tiny{$a_{i}$}}}
        \put(200,85){\vector(1,0){30}} \put(230,85){\vector(-1,0){30}} \put(215,75){\makebox(0,0)[b]{\tiny{$b_{i}$}}}

        \put(280,85){\vector(1,0){20}}  \put(300,85){\vector(-1,0){20}} \put(290,75){\makebox(0,0)[b]{\tiny{$a_{i+1}$}}}
        \put(300,85){\vector(1,0){40}} \put(340,85){\vector(-1,0){40}} \put(320,75){\makebox(0,0)[b]{\tiny{$b_{i+1}$}}}

        % Les probabilit\'es d'inclusion
        \put(60,170){\vector(1,0){60}}  \put(120,170){\vector(-1,0){60}} \put(90,175){\makebox(0,0)[b]{\tiny{$\phi_{2i-2}$}}}
        \put(120,170){\vector(1,0){50}}  \put(170,170){\vector(-1,0){50}} \put(145,175){\makebox(0,0)[b]{\tiny{$\phi_{2i-1}$}}}
        \put(170,170){\vector(1,0){60}}  \put(230,170){\vector(-1,0){60}} \put(200,175){\makebox(0,0)[b]{\tiny{$\phi_{2i}$}}}
        \put(230,170){\vector(1,0){50}}  \put(280,170){\vector(-1,0){50}} \put(255,175){\makebox(0,0)[b]{\tiny{$\phi_{2i+1}$}}}
        \put(280,170){\vector(1,0){60}}  \put(340,170){\vector(-1,0){60}} \put(310,175){\makebox(0,0)[b]{\tiny{$\phi_{2i+2}$}}}

        % Les micro-strates

        \put(30,65){\line(1,0){90}} \put(120,65){\line(0,1){10}} \put(60,50){\makebox(0,0)[b]{$U_{i-1}$}}
        \put(20,65){\line(1,0){5}}
        \put(10,65){\line(1,0){5}}
        \put(0,65){\line(1,0){5}}

        \put(60,30){\line(1,0){170}} \put(60,30){\line(0,1){10}} \put(230,30){\line(0,1){10}} \put(150,15){\makebox(0,0)[b]{$U_{i}$}}
        \put(170,65){\line(1,0){170}} \put(170,65){\line(0,1){10}} \put(340,65){\line(0,1){10}} \put(250,50){\makebox(0,0)[b]{$U_{i+1}$}}

        \put(280,30){\line(1,0){100}} \put(280,30){\line(0,1){10}} \put(355,15){\makebox(0,0)[b]{$U_{i+2}$}}
        \put(385,30){\line(1,0){5}}
        \put(395,30){\line(1,0){5}}
        \put(405,30){\line(1,0){5}}
        \end{picture}

        \caption{Inclusion probabilities and cross-border units in microstrata $U_{i}$ and $U_{i+1}$
        for population $U_c$} \label{nota:clust}
        \end{figure}
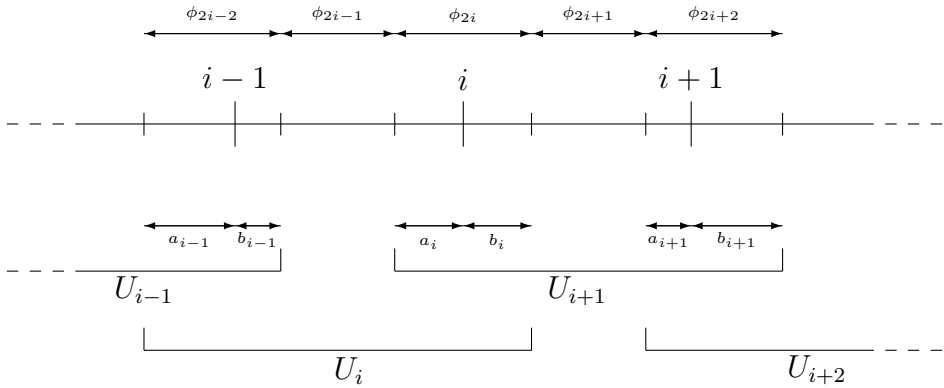
        \end{center}

        \begin{prop} \label{prop2}
          Chromy sampling with parameter $\pi$ in $U$ may be performed by two-stage sampling, with:
           \begin{enumerate}
             \item a first-stage selection of a sample $S_c$ of $n$ clusters by means of Chromy sampling with parameter $\phi$ in the population $U_c$,
             \item an independent second-stage selection inside each $u_i \in S_c$ of a sample $S_i$ of size $1$, with unit $k \in u_i$ selected with a probability $\pi_k/\phi_i$.
           \end{enumerate}
        \end{prop}

        \begin{proof}
        \noindent It is sufficient to prove that the transition probabilities given in Algorithm \ref{chr:samp} are the same under Chromy sampling with parameter $\pi$ and under the two-stage sampling procedure. We use the following notation: for any unit $k \in U$, recall that $I_k$ is the sample membership indicator under Chromy sampling with parameter $\pi$; for any cluster $u_i \in U_c$, $J_{c,i}$ is the sample membership indicator under Chromy sampling with parameter $\phi$; for any unit $k \in U$, $J_k$ is the sample membership indicator under the two-stage procedure. We first note that, by definition of the two-stage procedure, we have
            \begin{eqnarray} \label{pprop2:eq1}
              J_{c,i} & = & \sum_{k \in u_i} J_k \textrm{ for any } u_i \in U_c.
            \end{eqnarray}

        \noindent We now consider the case when $k \in U$ is a cross-border unit, $k_i$ say. The corresponding cluster is $u_{2i}$, and in such case
            \begin{eqnarray} \label{pprop2:eq2}
              J_{k_i} = 1 & \Leftrightarrow & J_{c,2i} = 1.
            \end{eqnarray}
        We obtain successively
            \begin{eqnarray} \label{pprop2:eq3}
              Pr\left(J_{k_i}=1 \left| \sum_{l=1}^{k_i-1} J_l \right.\right) & = & Pr\left(J_{k_i}=1 \left| \sum_{j=1}^{2i-1} J_{c,j} \right.\right) \textrm{ from equation (\ref{pprop2:eq1})} \nonumber \\
                                                     & = & Pr\left(J_{c,2i}=1 \left| \sum_{j=1}^{2i-1} J_{c,j} \right.\right) \textrm{ from equation (\ref{pprop2:eq2})} \nonumber \\
                                                     & = & \left\{\begin{array}{ll}
                                                                    1 & \textrm{if } \sum_{j=1}^{2i-1} J_{c,j} = i-1,\\
                                                                    \frac{b_i}{1-a_i} & \textrm{if } \sum_{j=1}^{2i-1} J_{c,j} = i,\\
                                                                  \end{array} \right. \nonumber \\
                                                     &   & \textrm{ from Step 2.b in Algorithm \ref{chr:samp}} \nonumber \\
                                                     & = & Pr\left(I_{k_i}=1 \left| \sum_{l=1}^{k_i-1} I_k \right.\right),
            \end{eqnarray}
        where the last line is obtained again from Algorithm \ref{chr:samp}. \\

        \noindent We now consider the case when $k$ is not a cross-border unit, and belongs to the cluster $u_{2i-1}$, say. We begin by computing the quantities
            \begin{eqnarray*} %\label{pprop2:eq4}
              Pr\left(J_k=1\left|\sum_{j=1}^{2i-2} J_{c,j} , \sum_{l=k_{i-1}+1}^{k-1} J_l \right.\right).
            \end{eqnarray*}
        Note that from Proposition \ref{prop1}, $\sum_{j=1}^{2i-2} J_{c,j}$ may only take the values $i-1$ and $i$, and since we select at most one unit $l$ inside $u_{2i-1}$, $\sum_{l=k_{i-1}+1}^{k-1} J_l$ may only take the values $0$ and $1$. If $\sum_{l=k_{i-1}+1}^{k-1} J_l=1$, we have $J_k=0$ since we select at most one unit $l$ inside $u_{2i-1}$. Therefore
            \begin{eqnarray} \label{pprop2:eq5}
              Pr\left(J_k=1\left|\sum_{j=1}^{2i-2} J_{c,j}=i-1 , \sum_{l=k_{i-1}+1}^{k-1} J_l=1 \right.\right)  & = & 0.
            \end{eqnarray}
        If $\sum_{j=1}^{2i-2} J_{c,j}=i$, Algorithm \ref{chr:samp} implies that the cluster $u_{2i-1}$ and therefore $k$ may not be selected. Therefore
            \begin{eqnarray} \label{pprop2:eq6}
              Pr\left(J_k=1\left|\sum_{j=1}^{2i-2} J_{c,j}=i , \sum_{l=k_{i-1}+1}^{k-1} J_l=0 \right.\right)  & = & 0.
            \end{eqnarray}
        For the same reason, we may not have simultaneously $\sum_{j=1}^{2i-2} J_{c,j}=i$ and $\sum_{l=k_{i-1}+1}^{k-1} J_l=1$. Finally, we have
            \begin{eqnarray} \label{pprop2:eq7}
               &  & Pr\left(J_k=1\left|\sum_{j=1}^{2i-2} J_{c,j}=i-1 , \sum_{l=k_{i-1}+1}^{k-1} J_l=0 \right.\right) \nonumber \\
              & =  & \frac{Pr\left(J_k=1,\sum_{l=k_{i-1}+1}^{k-1} J_l=0 \left| \sum_{j=1}^{2i-2} J_{c,j}=i-1 \right.\right)}{Pr\left(\sum_{l=k_{i-1}+1}^{k-1} J_l=0 \left| \sum_{j=1}^{2i-2} J_{c,j}=i-1 \right.\right)} \nonumber \\
              & = & \frac{Pr\left(J_k=1 \left| \sum_{j=1}^{2i-2} J_{c,j}=i-1 \right.\right)}{Pr\left(\sum_{l=k_{i-1}+1}^{k-1} J_l=0 \left| \sum_{j=1}^{2i-2} J_{c,j}=i-1 \right.\right)},
            \end{eqnarray}
        where the last line in (\ref{pprop2:eq7}) follows from the fact that if $J_k=1$, we necessarily have $\sum_{l=k_{i-1}+1}^{k-1} J_l=0$. We compute the numerator and the denominator in (\ref{pprop2:eq7}) separately. The numerator is
            \begin{eqnarray*} %\label{pprop2:eq8}
            & & Pr\left(J_k=1 \left| \sum_{j=1}^{2i-2} J_{c,j}=i-1 \right.\right) \nonumber \\
            & = & Pr\left(J_{c,2i-1}=1 \left| \sum_{j=1}^{2i-2} J_{c,j}=i-1 \right.\right) Pr\left(J_k=1 \left| \sum_{j=1}^{2i-2} J_{c,j}=i-1, J_{c,2i-1}=1 \right.\right) \nonumber \\
            & = & Pr\left(J_{c,2i-1}=1 \left| \sum_{j=1}^{2i-2} J_{c,j}=i-1 \right.\right) Pr\left(J_k=1 \left| J_{c,2i-1}=1 \right.\right).
            \end{eqnarray*}
        From Algorithm \ref{chr:samp}, we have $Pr\left(J_{c,2i-1}=1 \left| \sum_{j=1}^{2i-2} J_{c,j}=i-1 \right.\right)=\frac{1-b_{i-1}-a_i}{1-b_{i-1}}$, and from the definition of the two-stage procedure $Pr\left(J_k=1 \left| J_{c,2i-1}=1 \right.\right)=\frac{\pi_k}{1-b_{i-1}-a_i}$. This leads to
            \begin{eqnarray} \label{pprop2:eq9}
            Pr\left(J_k=1 \left| \sum_{j=1}^{2i-2} J_{c,j}=i-1 \right.\right) & = & \frac{\pi_k}{1-b_{i-1}}.
            \end{eqnarray}
        From the definition of the two-stage procedure, the denominator in (\ref{pprop2:eq7}) is
            \begin{eqnarray} \label{pprop2:eq10}
              Pr\left(\sum_{l=k_{i-1}+1}^{k-1} J_l=0 \left| \sum_{j=1}^{2i-2} J_{c,j}=i-1 \right.\right) & = & 1-Pr\left(\sum_{l=k_{i-1}+1}^{k-1} J_l=1 \left| \sum_{j=1}^{2i-2} J_{c,j}=i-1 \right.\right) \nonumber \\
              & = & 1-\frac{\sum_{l=k_{i-1}+1}^{k-1} \pi_l}{1-b_{i-1}}.
            \end{eqnarray}
        From (\ref{pprop2:eq7}), (\ref{pprop2:eq9}) and (\ref{pprop2:eq10}), we obtain
            \begin{eqnarray} \label{pprop2:eq11}
            Pr\left(J_k=1\left|\sum_{j=1}^{2i-2} J_{c,j}=i-1 , \sum_{l=k_{i-1}+1}^{k-1} J_l=0 \right.\right) & = & \frac{\pi_k}{1-b_{i-1}-\sum_{l <k \in u_{2i-1}} \pi_l} \nonumber \\
                                                                                                             & = & \frac{V_k^F-V_{k-1}^F}{1-V_{k-1}^F}.
            \end{eqnarray}
        From (\ref{pprop2:eq5}), (\ref{pprop2:eq6}) and (\ref{pprop2:eq11}), we obtain
            \begin{eqnarray} \label{pprop2:eq12}
            Pr\left(J_k=1\left|\sum_{j=1}^{2i-2} J_{c,j}+\sum_{l=k_{i-1}+1}^{k-1} J_l=i-1 \right.\right) & = & \frac{V_k^F-V_{k-1}^F}{1-V_{k-1}^F}, \nonumber \\
            Pr\left(J_k=1\left|\sum_{j=1}^{2i-2} J_{c,j}+\sum_{l=k_{i-1}+1}^{k-1} J_l=i \right.\right) & = & 0.
            \end{eqnarray}
        From equation (\ref{pprop2:eq1}), we have
            \begin{eqnarray} \label{pprop2:eq13}
              \sum_{j=1}^{2i-2} J_{c,j}+\sum_{l=k_{i-1}+1}^{k-1} J_l & = & \sum_{l=1}^{k_{i-1}} J_l+\sum_{l=k_{i-1}+1}^{k-1} J_l =\sum_{l=1}^{k-1} J_l,
            \end{eqnarray}
        and from (\ref{pprop2:eq12}), this implies that
            \begin{eqnarray} \label{pprop2:eq14}
            Pr\left(J_k=1\left|\sum_{l=1}^{k-1} J_l \right.\right) & = & \left\{\begin{array}{ll}
                                                                                  \frac{V_k^F-V_{k-1}^F}{1-V_{k-1}^F} & \textrm{if } \sum_{l=1}^{k-1} J_l=i-1, \\
                                                                                  0 & \textrm{if } \sum_{l=1}^{k-1} J_l=i,
                                                                                \end{array}
             \right. \nonumber \\
             & = & Pr\left(I_k=1\left|\sum_{l=1}^{k-1} I_l \right.\right),
            \end{eqnarray}
        where the last line in (\ref{pprop2:eq14}) follows from Step 2.b in Algorithm \ref{chr:samp}. This completes the proof.
        \end{proof}

    \subsection{Ordered sample} \label{appC2}

    \noindent We use the same notation as in Section \ref{appC1}, and consider a sample $S_c$ selected in $U_c$ by means of Chromy sampling with parameter $\phi$. In this case, we let $V_{c,i}=\sum_{j=1}^{i} \phi_j$ denote the cumulated inclusion probabilities up to unit $u_i$, with $V_{c,0}=0$. The integer part of $V_{c,i}$ is denoted as $V_{c,i}^I$, and the difference between $V_{c,i}$ and its integer part $V_{c,i}^I$ is the fractional part, denoted as $V_{c,i}^F$. \\

    Let us denote by $X_1<\ldots<X_n$ the selected units, ranked with respect to the natural order in the population $U$. The transition probabilities between the ranked selected units are given in Proposition \ref{prop3}.

    \begin{prop} \label{prop3}
    Let $S_c$ denote a sample selected by Chromy sampling with parameter $\phi$ in the population $U_c$. The transition probabilities between the ordered sampled units $X_1<\ldots<X_n$ are:
            \begin{eqnarray*}
              Pr(X_{i+1}=u_j|X_{i}=u_{2i-2}) & = & \left\{\begin{array}{ll}
                                                                               \frac{b_i}{1-a_i} & \textrm{if } j=2i, \\
                                                                               \frac{(1-b_i-a_{i+1})(1-a_i-b_{i})}{(1-a_i)(1-b_{i})} & \textrm{if } j=2i+1, \\
                                                                               \frac{a_{i+1}(1-a_i-b_i)}{(1-a_i)(1-b_i)} & \textrm{if } j=2i+2,
                                                                               \end{array}
               \right. \\
              Pr(X_{i+1}=u_j|X_{i}=u_{2i-1}) & = & \left\{\begin{array}{ll}
                                                                               \frac{b_i}{1-a_i} & \textrm{if } j=2i, \\
                                                                               \frac{(1-b_i-a_{i+1})(1-a_i-b_{i})}{(1-a_i)(1-b_{i})} & \textrm{if } j=2i+1, \\
                                                                               \frac{a_{i+1}(1-a_i-b_i)}{(1-a_i)(1-b_i)} & \textrm{if } j=2i+2,
                                                                               \end{array}
               \right. \\
              Pr(X_{i+1}=u_j|X_{i}=u_{2i})                           & = & \left\{\begin{array}{ll}
                                                                               \frac{1-b_i-a_{i+1}}{1-b_{i}} & \textrm{if } j=2i+1, \\
                                                                               \frac{a_{i+1}}{1-b_i} & \textrm{if } j=2i+2.
                                                                               \end{array}
               \right.
            \end{eqnarray*}
    \end{prop}

    \begin{proof}
    \noindent We first consider the case when $X_i=u_{2i-2}$, which is equivalent to $\sum_{j=1}^{2i-2} J_{c,j}=i$. Therefore,
        \begin{eqnarray} \label{pprop3:eq1}
          Pr(X_{i+1}=u_j|X_{i}=u_{2i-2}) & = & Pr\left(X_{i+1}=u_j \left|\sum_{j=1}^{2i-2} J_{c,j}=i \right. \right) \nonumber \\
                                         & = & Pr\left(X_{i+1}=u_j \left|\sum_{j=1}^{2i-1} J_{c,j}=i \right. \right),
        \end{eqnarray}
    where the second line in (\ref{pprop3:eq1}) follows from the fact that, from Step 2.a of Algorithm \ref{chr:samp}, $\sum_{j=1}^{2i-2} J_{c,j}=i$ implies that $J_{c,2i-1}=0$. If $j=2i$:
        \begin{eqnarray} \label{pprop3:eq2}
          Pr(X_{i+1}=u_{2i}|X_{i}=u_{2i-2}) & = & Pr\left(X_{i+1}=u_{2i} \left|\sum_{j=1}^{2i-1} J_{c,j}=i \right. \right) \nonumber \\
                                            & = & Pr\left(J_{c,2i}=1 \left|\sum_{j=1}^{2i-1} J_{c,j}=i \right. \right) \nonumber \\
                                            & = & \frac{V_{c,2i}^F}{V_{c,2i-1}^F} = \frac{b_i}{1-a_i}.
        \end{eqnarray}
    If $j=2i+1$:
        \begin{eqnarray} \label{pprop3:eq3}
          Pr(X_{i+1}=u_{2i+1}|X_{i}=u_{2i-2}) & = & Pr\left(X_{i+1}=u_{2i+1} \left|\sum_{j=1}^{2i-1} J_{c,j}=i \right. \right) \nonumber \\
                                            & = & Pr\left(J_{c,2i+1}=1,J_{c,2i}=0 \left|\sum_{j=1}^{2i-1} J_{c,j}=i \right. \right) \nonumber \\
                                            & = & Pr\left(J_{c,2i+1}=1 \left|\sum_{j=1}^{2i} J_{c,j}=i \right. \right)
                                                  Pr\left(J_{c,2i}=0 \left|\sum_{j=1}^{2i-1} J_{c,j}=i \right. \right)\nonumber \\
                                            & = & \left(\frac{V_{c,2i+1}^F-V_{c,2i}^F}{1-V_{c,2i}^F}\right) \left(1- \frac{b_i}{1-a_i} \right) \nonumber \\
                                            & = & \frac{(1-b_i-a_{i+1})(1-a_i-b_{i})}{(1-a_i)(1-b_{i})},
        \end{eqnarray}
     where the last but one line in (\ref{pprop3:eq3}) follows from Step 2.a of Algorithm \ref{chr:samp} and from equation (\ref{pprop3:eq2}). If $j=2i+2$, we obtain similarly:
        \begin{eqnarray} \label{pprop3:eq4}
          Pr(X_{i+1}=u_{2i+2}|X_{i}=u_{2i-2}) & = & Pr\left(X_{i+1}=u_{2i+2} \left|\sum_{j=1}^{2i-1} J_{c,j}=i \right. \right) \nonumber \\
                                            & = & Pr\left(J_{c,2i+2}=1,J_{c,2i+1}=0,J_{c,2i}=0 \left|\sum_{j=1}^{2i-1} J_{c,j}=i \right. \right) \nonumber \\
                                            & = & Pr\left(J_{c,2i+2}=1 \left|\sum_{j=1}^{2i+1} J_{c,j}=i \right. \right)
                                                  Pr\left(J_{c,2i+1}=0 \left|\sum_{j=1}^{2i} J_{c,j}=i \right. \right)\nonumber \\
                                            & \times & Pr\left(J_{c,2i}=0 \left|\sum_{j=1}^{2i-1} J_{c,j}=i \right. \right)\nonumber \\
                                            & = & 1 \times \left(1- \frac{1-a_{i+1}-b_i}{1-b_i} \right) \left(1- \frac{b_i}{1-a_i} \right) \nonumber \\
                                            & = & \frac{a_{i+1}(1-a_i-b_i)}{(1-a_i)(1-b_i)}.
        \end{eqnarray}
     This gives the first equation in Proposition \ref{prop3}. \\

     \noindent Now, we consider the case when $X_i=u_{2i-1}$. We have
        \begin{eqnarray} \label{pprop3:eq5}
          X_i=u_{2i-1} & \Rightarrow & \sum_{j=1}^{2i-2} J_{c,j}=i-1 \textrm{ and } J_{c,2i-1}=1 \nonumber \\
                       & \Rightarrow & \sum_{j=1}^{2i-1} J_{c,j}=i.
        \end{eqnarray}
     Since at each step of Chromy sampling, the conditional probabilities only depend on the number of units already selected, this leads to
        \begin{eqnarray} \label{pprop3:eq6}
          Pr(X_{i+1}=u_j|X_{i}=u_{2i-1}) & = & Pr\left(X_{i+1}=u_j \left|\sum_{j=1}^{2i-1} J_{c,j}=i \right. \right),
        \end{eqnarray}
     which is identical to equation (\ref{pprop3:eq1}). Therefore, the second equation in Proposition \ref{prop3} follows. \\

     \noindent Finally, we consider the case when $X_i=u_{2i}$, which is equivalent to $\sum_{j=1}^{2i} J_{c,j}=i$. If $j=2i+1$:
        \begin{eqnarray} \label{pprop3:eq7}
          Pr(X_{i+1}=u_{2i+1}|X_{i}=u_{2i}) & = & Pr\left(X_{i+1}=u_{2i+1} \left|\sum_{j=1}^{2i} J_{c,j}=i \right. \right) \nonumber \\
                                            & = & Pr\left(J_{c,2i+1}=1 \left|\sum_{j=1}^{2i} J_{c,j}=i \right. \right) \nonumber \\
                                            & = & \frac{V_{c,2i+1}^F-V_{c,2i}^F}{1-V_{c,2i}^F} = \frac{1-a_{i+1}-b_{i}}{1-b_{i}}.
        \end{eqnarray}
     If $j=2i+2$:
        \begin{eqnarray} \label{pprop3:eq8}
          Pr(X_{i+1}=u_{2i+2}|X_{i}=u_{2i}) & = & Pr\left(X_{i+1}=u_{2i+2} \left|\sum_{j=1}^{2i} J_{c,j}=i \right. \right) \nonumber \\
                                            & = & Pr\left(J_{c,2i+2}=1,J_{c,2i+2}=0 \left|\sum_{j=1}^{2i} J_{c,j}=i \right. \right) \nonumber \\
                                            & = & Pr\left(J_{c,2i+2}=1 \left|\sum_{j=1}^{2i+1} J_{c,j}=i \right. \right)
                                                  Pr\left(J_{c,2i+1}=0 \left|\sum_{j=1}^{2i} J_{c,j}=i \right. \right) \nonumber \\
                                            & = & 1 \times \left(1-\frac{1-a_{i+1}-b_{i}}{1-b_{i}}\right) = \frac{a_{i+1}}{1-b_i}.
        \end{eqnarray}
     This completes the proof.
    \end{proof}

    \subsection{Proof of Theorem \ref{theo1}} \label{appC3}

    \noindent From Lemma 3.1 in \cite{cha:12} and our Proposition \ref{prop2}, Chromy sampling and ordered pivotal sampling have the same two-stage characterization. It is therefore sufficient to prove that they lead to the same sampling design when sampling with parameter $\phi$ in the clustered population $U_c$. However, from equations (4.2)-(4.4) in \cite{cha:12} and our Proposition \ref{prop3}, both Chromy sampling and ordered pivotal sampling have the same transition probabilities between ordered sampled units, which means that the induced sampling designs are identical.

\end{document}